\documentclass[leqno,11pt]{amsart}

\usepackage{amssymb}
\usepackage{color}
\usepackage{latexsym}
\usepackage{amscd}
\usepackage{enumerate} 
\theoremstyle{plain}
\newcounter{theo}
\newtheorem{thm}[theo]{Theorem}
\newtheorem{prop}[theo]{Proposition}
\newtheorem{lem}[theo]{Lemma}

\theoremstyle{definition}
\newtheorem{rem}[theo]{Remark}

\title[a counterexample to the $2-$jet determination Chern-Moser Theorem in higher codimension]{a  counterexample to the $2-$jet determination Chern-Moser Theorem in higher codimension}
\author{Francine Meylan}

\newcommand{\C}{\mathbb{C}}

\newcommand{\Imm}{\mathrm{Im}\,}

\begin{document}

\maketitle

\begin{abstract}
One  constructs an example of a  generic  quadratic  submanifold of  codimension $5$ in $\Bbb C^9$ which  admits a real analytic 
infinitesimal CR automorphism with homogeneous polynomial coefficients of degree $3.$
\end{abstract}

\section{Introduction}
Let $M$ be a real-analytic  submanifold of $\C^N$ of codimension $d.$  Consider the set of germs of biholomorphisms $F$  at a point $p\in M$  such that $F(M)\subset M$. By the   work of  Chern and Moser\cite{CM}, if the  codimension $d=1,$  every such $F$ is uniquely determined by its derivatives  at $p  $ provided that its  Levi map  at $p$ is non-degenerate. See also Cartan and Tanaka ( \cite{Ca}, 
 \cite{Ta}).


\begin{thm} \cite{CM}
Let $M$ be a real-analytic hypersurface through a point $p$ in $\C^N$ with non-degenerate Levi form at $p$. Let $F$, $G$ be two germs of biholomorphic maps preserving $M$. Then, if $F$ and $G$ have the same 2-jets at $p$, they coincide.
\end{thm}
Note that the result becomes false without any hypothesis on the Levi form (See for instance  \cite{bl-me1}). A generalization of this Theorem to  real-analytic submanifolds  $M$ of higher codimension $d>1$ has been proposed by Beloshapka in \cite{Be2} (and quoted many times by several authors)  under the hypothesis that  $M$ is  Levi generating (or equivalently of finite type with $2$ the only H\"ormander number) with  non-degenerate Levi map. Unfortunately, an error has been discovered  and explained in \cite{bl-me1}. In this short note, one  constructs an example of a  generic (Levi generating  with non-degenerate Levi map)   quadratic  submanifold  that admits  an element in its stability group which has the same $2-$jet as the identity map but is not the identity map. In addition, this example is Levi non-degenerate in the sense of Tumanov. One  points out that if $M$ is strictly pseudoconvex, that is Levi non-degenerate  in the sense of Tumanov with a positivity condition,  then by a recent result of Tumanov\cite{tu3}, the  $2$-jet determination result holds in any codimension.


One also  points out that
finite jet determination problems for   submanifolds has attracted  much attention.
 One  refers  in particular  to the papers of Zaitsev \cite{Za}, Baouendi, Ebenfelt and Rothschild \cite{BER1}, Baouendi, Mir and Rothschild \cite{BMR}, Ebenfelt, Lamel and Zaitsev \cite{eb-la-za},  Lamel and Mir \cite{la-mi}, Juhlin \cite{ju}, Juhlin and Lamel \cite{ju-la}, Mir and Zaitsev \cite{mi-za} in the real analytic case, Ebenfelt \cite{eb}, Ebenfelt and Lamel \cite{eb-la}, Kim and Zaitsev \cite{ki-za}, Kolar, the  author and Zaitsev \cite{KMZ}  in the
$ \mathcal{C}^\infty$ case,  Bertrand and Blanc-Centi \cite{be-bl}, Bertrand, Blanc-Centi  and the  author \cite {be-bl-me}, Tumanov \cite{tu3}  in the finitely smooth case.
\bigskip

\section{The  Example}
Let $M \subseteq \C^{9}$ be the  real submanifold of (real) codimension $5$ through $0$ given  in the coordinates $(z,w)=(z_1, \dots,z_4, w_1, \dots, w_5)  \in \C^{9},$ by 
\begin{equation}\label{Fe}
\begin{cases}
\Imm w_1= P_1 (z, \bar z)=z_1 \overline{ z_2} +  z_2 \overline{ z_1} \\
\Imm w_2=P_2 (z, \bar z)= -i z_1 \overline{ z_2} + i z_2 \overline{ z_1} \\
 \Imm w_3=P_3 (z, \bar z) = z_3 \overline{ z_2} + z_4 \overline{ z_1} + z_2 \overline{ z_3} + z_1 \overline{ z_4} \\
 \Imm w_4=P_4 (z, \bar z)= z_1 \overline{ z_1}  \\
 \Imm w_5=P_5 (z, \bar z) =z_2 \overline{ z_2}  \\
\end{cases}
\end{equation}

The matrices corresponding to the $P_i's$ are 
$$A_1=\left(\begin{array}{cccc}0&1&0&0\\1&0&0&0\\0&0&0&0\\0&0&0&0\end{array}\right)\qquad 
A_2=\left(\begin{array}{cccc}0&-i&0&0\\i&0&0&0\\0&0&0&0\\0&0&0&0\end{array}\right)\qquad 
A_3=\left(\begin{array}{cccc}0&0&0&1\\0&0&1&0\\0&1&0&0\\1&0&0&0\end{array}\right)\qquad $$
$$
A_4=\left(\begin{array}{cccc}1&0&0&0\\0&0&0&0\\0&0&0&0\\0&0&0&0\end{array}\right)\qquad 
A_5=\left(\begin{array}{cccc}0&0&0&0\\0&1&0&0\\0&0&0&0\\0&0&0&0\end{array}\right)\qquad 
$$

\begin{lem} The following holds:
\begin{enumerate}
\item the $A_i's$ are linearly independent, \\
\item the $A_i's$ satisfy the condition of Tumanov, that is, there is $c \in \Bbb R^d$ such that $\det{\sum c_jA_j}\ne 0.$
\end{enumerate}
\end{lem}
\begin{prop}
The real submanifold  $M$ given by \eqref{Fe} is Levi generating at $0,$ that is, of finite type with $2$ the only H\"ormander number, and its Levi map is  non-degenerate.
\end{prop}
\begin{proof}
This follows for instance from Proposition 8, Lemma 3 and Remark 4 in \cite{bl-me1}.
\end{proof}
\begin{rem} The following identity between  the $P_i's$ holds:
\begin{equation}\label{fer1}
{P_1}^2 + {P_2}^2 -4 P_4P_5=0.
\end{equation}
\end{rem}
The following  holomorphic vectors fields  are in $hol(M,0),$ the set of germs of real-analytic infinitesimal CR automorphisms at $0.$

\begin{enumerate}
\item $ X:=i(z_1
\dfrac{\partial}{\partial {z_3}}+z_2 \dfrac{\partial}{\partial {z_4}})$
\item $Y:= i(-iz_1\dfrac{\partial}{\partial {z_3}}+iz_2 \dfrac{\partial}{\partial {z_4}})$
\item $Z:= i(z_1\dfrac{\partial}{\partial {z_4}})$
\item $U:= i(z_2\dfrac{\partial}{\partial {z_3}})$
\end{enumerate}

\begin{lem} Let $P=(P_1, \dots, P_4).$
The following holds:
\begin{enumerate}
\item $X(P)=(0,0,iP_1,0,0)$
\item $Y(P)=(0,0,iP_2,0,0)$
\item $Z(P)=(0,0,iP_4,0,0)$
\item $U(P)=(0,0,iP_5,0,0).$
\end{enumerate}
\end{lem}

\begin{lem} The following identities  hold:

\begin{enumerate}
\item $P_1 (-Y(P)) +P_2 (X(P)=0$
\item $P_1X(P) + P_2(Y(P) + P_5 (-2Z(P)) + P_4 (-2U(P)) =0$
\item $P_2(-2Z(P)) + P_4(2Y(P) =0$
\item $P_2 (-2U(P)) + P_5 (2Y(P) =0$
\end{enumerate}
\end{lem}
With the help of the Lemmata, one obtains

\begin{thm}  Let  $Y_0=-Y, \ \ Y_1=2Y, \ \ Z_1 = -2Z, \ \ U_1= -2U.$

The holomorphic vector field $T$ defined by
\begin{equation}
T= \dfrac{1}{2}{w_1}^2Y_0 + \dfrac{1}{2}{w_2}^2 Y +
w_1w_2 X + w_2w_5 Z_1 +w_2w_4 U_1 + w_4w_5 Y_1
\end{equation} is in  $ hol (M,0).$


Hence $2-$jet determination does not hold for germs of  biholomorphisms sending  $M$ to $M.$
\end{thm}
\begin{rem}
Notice  that the bound for the   number $k$ of jets  needed to determine uniquely any germ of  biholomorphism sending $M$ to $M$   is  $$k= (1+ \text{codim} \  M),$$  $M$ beeing a generic (Levi generating with non-degenerate Levi map) real-analytic submanifold: see Theorem 12.3.11, page 361 in \cite{BER}. One points out that  Zaitsev  obtained the bound  $k= 2(1+ \text{codim} \  M)$ in \cite{Za}.
\end{rem}


\bigskip
\begin{thebibliography}{12}
\itemsep=2pt

\bibitem{BER} M.S. Baouendi, P. Ebenfelt, L.P. Rothschild: \textit{Real Submanifolds in Complex Space and their Mappings}. Princeton University Press, 1999.

\bibitem{BER1} M.S. Baoudendi, P. Ebenfelt, L.P. Rothschild, {\it CR automorphisms of real analytic CR manifolds in complex space}, Comm. Anal. Geom. {\bf 6}  (1998), 291-315.
  

\bibitem{Be2} V.K. Beloshapka: A uniqueness theorem for automorphisms of a nondegenerate surface in a complex space. (Russian) {\it Mat. Zametki} {\bf 47}(3) 17-22, 1990; transl. {\it Math. Notes} {\bf 47}(3) 239-243, 1990.

\bibitem{be-bl} F. Bertrand, L. Blanc-Centi, {\it Stationary holomorphic discs and finite jet determination problems},  
Math. Ann. {\bf 358} (2014), 477-509.

\bibitem{be-bl-me} F. Bertrand, L. Blanc-Centi, F. Meylan, {\it Stationary discs and finite jet determination for non-degenerate generic real submanifolds}
Adv. Math. {\bf 343} (2019), 910-934. 
\bibitem{bl-me1} L. Blanc-Centi, F. Meylan, {\it On  nondegeneracy conditions  for the Levi map in higher codimension: a Survey}, preprint, arXiv:1711.11481.


\bibitem{BMR} M.S. Baouendi, N. Mir, L.P. Rothschild,    
 \textit{Reflection Ideals and mappings between generic submanifolds in complex space}, 
J. Geom. Anal.  \textbf{12} 
  (2002),   543-580.
  
 \bibitem{Ca} Cartan, E:\textit { Sur la g\' eom\' etrie pseudo-conforme des hypersurfaces de deux variables complexes}  Ann.
Math. Pura Appl. 11 (1932), 17-90. 
\bibitem{CM} S.S. Chern, J.K. Moser: Real hypersurfaces in complex manifolds, {\it Acta Math.} {\bf 133}(3-4) 219-271 (1974).

\bibitem{eb} P. Ebenfelt, {\it Finite jet determination of holomorphic mappings at the boundary}, 
Asian J. Math. {\bf 5} (2001), 637-662.

\bibitem{eb-la}P. Ebenfelt, B. Lamel, {\it Finite jet determination of CR embeddings}, 
J. Geom. Anal. {\bf 14} (2004), 241-265.


\bibitem{eb-la-za} P. Ebenfelt, B. Lamel, D. Zaitsev, {\it Finite jet determination of local analytic CR automorphisms and their parametrization by 2-jets in the finite type case}, Geom. Funct. Anal.
 {\bf 13} (2003), 546-573.



\bibitem{ju}  R. Juhlin, {\it Determination of formal CR mappings by a finite jet}, Adv. Math. {\bf 222} (2009), 1611-1648.
 
\bibitem{ju-la} R. Juhlin,  B. Lamel, {\it Automorphism groups of minimal real-analytic CR manifolds},  J. Eur. Math. Soc. (JEMS) {\bf 15} (2013), 509-537.

\bibitem{ki-za} S.-Y. Kim,  D. Zaitsev, {\it Equivalence and embedding problems for CR-structures of any 
codimension}, Topology {\bf 44} (2005), 557-584.


\bibitem{KMZ} M. Kol\'a\v r, F.  Meylan, D. Zaitsev, {\it Chern-Moser operators and polynomial models in CR geometry},  Adv. Math.  \textbf{263} (2014), 321-356.


\bibitem{la-mi} B. Lamel,  N. Mir, {\it Finite jet determination of CR mappings}, Adv. Math. {\bf 216} (2007), 153-177.
  
 \bibitem{mi-za}  N. Mir,  D. Zaitsev, {\it Unique jet determination and extension of germs of CR maps into spheres}, preprint.

\bibitem{Ta} Tanaka, N. {\it On the pseudo-conformal geometry of hupersurfaces of the space of n complex variables}. J.
Math. Soc. Japan 14 (1962), 397-429.
\bibitem{tu3} A. Tumanov, {\it Stationary Discs and finite jet determination for CR mappings in higher codimension},  preprint, arXiv:1912.03782v1.


\bibitem{Za} D. Zaitsev: Germs of local automorphisms of real analytic CR structures and analytic dependence on the $k$-jets. {\it Math. Res. Lett.} {\bf 4}(6) 823-842, 1997.

\end{thebibliography}
\end{document}